\documentclass[11pt,a4paper,reqno]{amsart}

\usepackage[margin=1.0in]{geometry}
\usepackage{graphicx}
\usepackage{ amssymb }
\usepackage{ textcomp }
\usepackage[utf8]{inputenc}
\usepackage{hyperref,amsfonts}
\newtheorem{theorem}{Theorem}[section]

\newtheorem{proposition}[theorem]{Proposition}

\usepackage{upgreek}
\theoremstyle{definition}

\usepackage{amsmath}

\theoremstyle{remark}

\numberwithin{equation}{section}

\def\R{\mathbb{R}}

\def\Z{\mathbb{Z}}

\def\mc{\mathcal}

\def\lesim{\lesssim}

\def\beq{\begin{equation}}
\def\endeq{\end{equation}}

\theoremstyle{plain}
\newtheorem{thm}{Theorem}[section]

\newtheorem{claim}[thm]{Claim}

\newtheorem*{conj*}{Conjecture}
\newtheorem*{openproblem*}{Open Problem}

\theoremstyle{remark}

\begin{document}

\title[Wolff's inequality]{Remarks on Wolff's inequality for hypersurfaces}
\author{Shaoming Guo and Changkeun Oh}
\date{}
\maketitle

\begin{abstract}
We run an iteration argument due to Pramanik and Seeger \cite{PS}, to provide a proof of sharp decoupling inequalities for conical surfaces and for $k$-cones. These are extensions of results of \L aba and Pramanik \cite{LP1} to sharp exponents. 
\end{abstract}

%
%
\let\thefootnote\relax

\section{Statements of the results}\label{sec1}

For $n\ge 2$, let $L_0\subset \R^{n+1}$ be an affine subspace of dimension $n$ that does not pass through the origin. Let $E_0\subset L_0$ be a smooth compact surface of dimension $n-1$. Moreover, if $L_0$ is identified with $\mathbb{R}^n$ in a canonical manner, then we can assume that $E_0$ has a non-vanishing Gaussian curvature at every point. The surface $S$ given by 
$$
S=\{tx\in \R^{n+1}: x\in E_0; t\in [C_1, C_2]\}
$$
for some $0<C_1<C_2$ is called a \emph{conical surface} induced by $E_0$. For each $a \in S$, there exists a unique $b \in E_{0}$ such that $a=tb$ for some $t \in [C_{1},C_{2}]$. We denote by $\eta(a)$ the convex hull $[C_{1}b, C_{2}b]$ in $\mathbb{R}^{n+1}$, and call $\eta(a)$ the \textit{$1$-plane} at $a$.\\

Now we follow the approach of \L aba and Pramanik \cite{LP1} to introduce the notion of conical surfaces of higher co-dimensions.

Let $L_{0}$ be an $n$-dimensional linear subspace of $\mathbb{R}^{n+k}$. Let $v_{1}, v_{2}, \ldots,v_{k}$ be linearly independent vectors such that 
$$\{ x+c_{1}v_{1}+c_2 v_2+ \cdots+c_{k}v_{k} \in \mathbb{R}^{n+k} : x \in  L_{0}, \, (c_{1}, c_2, \ldots, c_{k}) \in \mathbb{R}^k \} = \mathbb{R}^{n+k}.$$
For each $i=1, 2, \ldots,k$, denote 
$$
L_{i}=L_{0}+v_{i}.
$$ 
For each $L_{i}$, we fix a bounded and convex solid $F_{i}$ such that $E_{i}:=\partial{F_{i}}$ is a $C^{\infty}$ surface and has non-vanishing Gaussian curvature at every point on it. Thus for each unit normal vector $x\in S^{n-1}$ in $\mathbb{R}^{n}$, each $E_{i}$ contains exactly one point $a_{i}$ such that $x$ is the outward normal vector to $E_{i}\subset L_{i}$ at $a_{i}$. We say that a $(k+1)$-tuple of points $(x_{0},\ldots,x_{k})$ is \textit{good} if $x_{i} \in E_{i}$ for every $0\le i\le k$, and if the outward unit normal vectors to $E_{i}$ at $x_{i}$ are the same. The \textit{k-cone} $S$ in $\mathbb{R}^{n+k}$ induced by the collection $\{E_{i}\}_{i=0}^{k}$ is defined by
$$S = \bigcup_{(x_{0},\ldots,x_{k}) : \textit{good}} \eta(x_{0},\ldots,x_{k}),$$
where $\eta(x_{0},\ldots,x_{k})$ denotes the convex hull generated by $x_{0},\ldots,x_{k}$ in $\mathbb{R}^{n+k}$. According to Lemma 7.1 in \cite{LP1}, each $a \in S$ belongs to $\eta(x_{0},\ldots,x_{k})$ for exactly one good $(k+1)$-tuple $(x_{0}, \ldots, x_{k})$. We will call $\eta(x_{0}, \ldots, x_{k})$ the \textit{k-plane} at $a$, and denote it by $\eta(a)$. \\

Let $S$ be a conical surface induced by $E_0$ or a \textit{k}-cone induced by $\{E_i\}_{i=0}^k$. For each $a \in S$, denote by $n_a$ the unit normal vector to $S$ at $a$. For a small number $\delta>0$, we denote by $\mc{M}_{\delta}$ a $\delta^{1/2}$-separated subset of $E_0$.  Moreover, denote by $\mathcal{N}_{\delta}S$ the $\delta$-neighborhood of $S$. Throughout the paper, we are interested in a covering of $\mc{N}_{\delta}S$ satisfying the following assumption.

\subsection*{Assumption (A)}
For each small $\delta >0$ and each $a \in S$, let $\Pi_{a,\delta}$ be a rectangular box centered at $a$, of dimensions $C\delta \times \underbrace{C\delta^{\frac{1}{2}} \times \cdots \times C\delta^{\frac{1}{2}}}_{(n-1) \text{ copies }}\times \underbrace{C \times \cdots \times C}_{k \text{ copies}}$, where the short direction is normal to $S$ at $a$, the long directions are parallel to the $k$-plane $\eta(a)$ at $a$, and the mid-length directions are tangent to $S$ at $a$ but perpendicular to the \textit{k}-plane $\eta(a)$. Then\\

\begin{description}
	\item[$A_1$] 	$C_{1}\Pi_{a,\delta} \subset \mathcal{N}_{\delta}S \cap 
	\{x\in \mathbb{R}^{n+k}: (x-a)\cdot n_a \leq \delta	\}$
	for some small constant $C_{1}>0$.\\
	\item[$A_2$]	$\{\Pi_{a,\delta}	\}_{a \in \mathcal{M}_{\delta}}$ forms a finitely overlapping covering of $\mathcal{N}_{\delta}S$. \\
	\item[$A_3$] For every $a \in \mathcal{M}_{\delta}$, there are at most $O(1)$ distinct $b \in \mathcal{M}_{\delta}$ such that $|n_a-n_b| \leq C_ {2}\delta^{1/2}$.
	\\
	\item[$A_4$] If $0< \delta \leq \sigma$ and if $
	\Pi_{a,\delta} \cap \Pi_{b,\sigma} \neq \emptyset$ for some $a,b \in S$, then $\Pi_{a,\delta} \subset C_{3}\Pi_{b,\sigma}$.	\\
\end{description}
This group of assumptions is identical to that in \cite{LP1}. The constants $C,C_{1},C_{2},C_{3}$ are independent of the parameter $\delta$ and the choice of $\mathcal{M}_{\delta}$.

Let $\Xi_{a}$ be a smooth function in $\mathbb{R}^{n+k}$ with 
$\|{\Xi_{a}}\|_{L^1(\mathbb{R}^{n+k})} \sim 1$ such that $\mathrm{supp}(\widehat{\Xi_{a}}) \subset \Pi_{a,\delta}$ and $\{ {\widehat{\Xi_{a}}}\}_{a \in \mathcal{M}_{\delta}} $ forms a smooth partition of unity of 
$\mathcal{N}_{\delta}S$.

\begin{theorem}\label{main1}
	Let $n \geq 2$ and $k \geq 1$.
	Let $S$ be a $k$-cone in $\mathbb{R}^{n+k}$. Under Assumption $(A)$, if $\mathrm{supp}(\hat{f}) \subset \mathcal{N}_{\delta}S$, then for $p \geq 2+\frac{4}{n-1}$, we have 
	\begin{equation}\label{p2-decoupling}
	\|{f}\|_{L^p(\mathbb{R}^{n+k})} \leq  C_{p,\epsilon} 
	\delta^{-\frac{n-1}{4}+\frac{n+1}{2p}-\epsilon}\biggl( \sum_{a \in \mathcal{M}_{\delta}} \|{\Xi_{a} * f}\|_{L^p(\mathbb{R}^{n+k})}^2 \biggr)^{\frac{1}{2}},
	\end{equation}
for every $\epsilon>0$.
\end{theorem}

By a standard interpolation, the above estimate \eqref{p2-decoupling} further implies 
\begin{equation}\label{p2-decoupling-1}
\|f\|_{L^p} \leq C_{p,\epsilon}  \delta^{-\epsilon} (\sum_{a\in \mc{M}_{\delta}} \|\Xi_{a}*f\|_{L^p}^2)^{1/2}
\end{equation}
for every $2 \leq p \leq 2+\frac{4}{n-1}$ and every $\epsilon>0$. Up to the arbitrarily small factor $\epsilon>0$, both \eqref{p2-decoupling} and \eqref{p2-decoupling-1} are sharp. For the sharpness we refer to the discussion in the introduction of the paper \cite{LP1}.\\

Theorem \ref{main1} involves $k$-cones. Recall that $k$-cones are generated by the boundaries $E_i$ of bounded and strictly convex bodies $F_i\subset L_i$ with $0\le i\le k$.  That means, for each $i$, if $L_i$ is identified with $\mathbb{R}^n$ in a canonical manner, then at every point on $E_i$, all the principle curvatures are positive. However, in the definition of a conical surface $S$ induced by $E_0$, we only assumed $E_0$ to have non-vanishing Gaussian curvatures. That means principle curvatures might have different signs. For conical surfaces, we prove 

\begin{theorem}\label{main2}
	Let $n \geq 2$.
	Let $S$ be a conical surface in $\mathbb{R}^{n+1}$. Under Assumption $(A)$, if $\mathrm{supp}(\hat{f}) \subset \mathcal{N}_{\delta}S$, then for $p \geq 2+\frac{4}{n-1}$, we have 
	\begin{displaymath}
	\|{f}\|_{L^p(\mathbb{R}^{n+1})} \leq C_{p,\epsilon} 
	\delta^{\frac{n}{p}-\frac{n-1}{2}-\epsilon}\biggl( \sum_{a \in \mathcal{M}_{\delta}} \|{\Xi_{a} * f}\|_{L^p(\mathbb{R}^{n+1})}^p \biggr)^{\frac{1}{p}},
	\end{displaymath}
for every $\epsilon>0$.
\end{theorem}

Theorem \ref{main1} and Theorem \ref{main2} are extensions of results in \L aba and Pramanik \cite{LP1} to sharp exponents. Our proof relies on an iteration argument and on results of Bourgain and Demeter \cite{BD1}, \cite{BD2}. This iteration argument was first used by Pramanik and Seeger \cite{PS}, and was later used by Bourgain and Demeter \cite{BD2} to obtain sharp decoupling estimates for the cone. For the prior developments on Wolff's inequalities, we refer to Wolff \cite{Wolff00}, \L aba and Wolff \cite{LW02}, Garrig\'os and Seeger \cite{GS09}, \cite{GS10}.\\

For $(\xi_{1},\ldots,\xi_{n}) \in \mathbb{R}^n$, we use the notation $\xi =(\xi_{1},\ldots,\xi_{n})$ and $\xi' = (\xi_{1},\ldots,\xi_{n-1})$. 
Throughout the paper, we write $A \lesssim B$ if $A \leq cB$ for some constant $c>0$, and $A \sim B$ if $c^{-1}A \leq B \leq cB$. The constant $c$ will in general depend on fixed parameters such as $p, n$ and sometimes on the variable parameter $\epsilon$ but not the parameter $\delta$.\\

{\bf Acknowledgements.} The first author would like to thank Ciprian Demeter for helpful discussions.
The second author would like to thank his advisor, Prof Jong-Guk Bak, for suggesting
this research topic and for many valuable discussions. \\

\section{Proof of Theorem 1.1}\label{sec2}

A truncated hyperbolic paraboloid ${H}_{v}^{n-1}$ in $\mathbb{R}^{n}$ is defined for $v=(v_{1},\ldots,v_{n-1}) \in (\mathbb{R}\setminus \{0\})^{n-1}$ as
\begin{displaymath}
{H}_{v}^{n-1} = \{ (\xi_{1}, \ldots, \xi_{n-1}, v_{1}\xi_{1}^2+\cdots +v_{n-1}\xi_{n-1}^2) : 
|\xi_{i}| \leq 1 \}.
\end{displaymath}
When $v_{i}=1$ for all $i$, we use $P^{n-1}$ instead of $H_{v}^{n-1}$.
We denote by $\mathcal{N}_{\delta}H_{v}^{n-1}$ the $\delta$-neighborhood of $H_{v}^{n-1}$.
Let $\mathcal{P}_{\delta}$ be a finitely overlapping cover of $\mathcal{N}_{\delta}H_{v}^{n-1}$ with $\delta \times \delta^{1/2} \times \cdots \times \delta^{1/2}$ rectangular boxes $\Pi_{a,\delta}'$ centered at $a$. Moreover, denote $\mathcal{M}_{\delta}=\{a:\Pi_{a,\delta}' \in \mathcal{P}_{\delta}	\}$. For each $a\in \mc{M}_{\delta}$, let $\Xi_{a}'$ be a smooth function in $\mathbb{R}^{n}$ with $\|\Xi_{a}'\|_{L^1(\mathbb{R}^{n})} \sim 1$ and $\mathrm{supp}(\widehat{\Xi'_{a}}) \subset \Pi_{a,\delta}'$ such that $\{\widehat{\Xi'_{a}}\}_{a\in \mathcal{M}_{\delta}}$ forms a smooth partition of unity of $\mathcal{N}_{\delta}H_{v}^{n-1}$. 

To prove Theorem \ref{main1}, we will use the following theorem due to Bourgain and Demeter.

\begin{theorem}[\cite{BD2}]\label{b-d}
	Denote $p_{0}=\frac{2(n+1)}{n-1}$. If $\mathrm{supp}(\hat{f}) \subset \mathcal{N}_{\delta}{P}^{n-1}$, then 
	\begin{displaymath}
	\|f\|_{L^{p_{0}}(\mathbb{R}^{n})} \lesssim_{\epsilon} \delta^{-\epsilon}\biggl(\sum_{a \in \mathcal{M}_{\delta}} \| \Xi_{a}' * f \|_{L^{p_{0}}(\mathbb{R}^{n})}^{2}\biggr)^\frac{1}{2},
	\end{displaymath}
for every $\epsilon>0$.
\end{theorem}
In the forthcoming proof of Theorem \ref{main1}, we consider only the endpoint $p_{0}=\frac{2(n+1)}{n-1}$. The estimate for the general range follows from the interpolation with the trivial estimate at $p=\infty$.
%
%
%
\subsection{}\label{subsection2.1}
In the first step of the proof, we will slice our surface into small pieces so that we can exploit local properties of a $k$-cone. Let $\{e_{i}\}_{i=1}^{n+k}$ be a collection of standard orthonormal bases in $\mathbb{R}^{n+k}$.
By a linear transformation, we may assume that $L_{0}=\mathrm{span}(e_{1},\ldots,e_{n})$ and $L_{i}=L_{0}+e_{n+i}$ for each $1\le i\le k$.

Fix a small parameter $\epsilon>0$. This $\epsilon$ is essentially the same as the one in the statement of Theorem \ref{main1}. We may also assume that $\epsilon^{-1}$ is a natural number. We define a sliced surface $\tilde{S}$ by
$$
\tilde{S}= S \cap (\mathbb{R}^{n} \times \{(\tau_{1},\ldots,\tau_{k}): c_{i} \leq \tau_{i} \leq c_{i}+4\delta^{\epsilon/2}	\}
),$$
for some $c_i$ with $1\le i\le k$.  We will prove the decoupling for the sliced surface $\tilde{S}$ first.

\begin{proposition}\label{1011prop2.2}
	If $\mathrm{supp}(\hat{f}) \subset \mathcal{N}_{\delta}\tilde{S}$, then
	$$\| f \|_{L^{p_{0}}(\mathbb{R}^{n+k})} \lesssim_{\epsilon} \delta^{-\epsilon}\biggl( \sum_{a \in \mathcal{M}_{\delta}} \|{\Xi_{a} * f}\|_{L^{p_{0}}(\mathbb{R}^{n+k})}^{2} \biggr)^{\frac{1}{2}}. $$
\end{proposition}

The desired decoupling inequalities for the surface $S$ can be deduced from Proposition \ref{1011prop2.2}. To see this, let $\{\hat{\psi}_{j}\}_{j \in \mathbb{Z}}$ be a partition of unity of $\mathbb{R}$ such that 
\begin{displaymath}
\|{\psi_{j}}\|_{L^1(\mathbb{R})} \sim 1 \text{ and }\mathrm{supp}(\hat{\psi}_{j}) \subset [(j-2)\delta^{\epsilon/2},(j+2)\delta^{\epsilon/2}].
\end{displaymath} 
For each $J=(j_1, ..., j_k)\in \Z^k$, we define 
$$
f_{J}(x,t)=\int_{\mathbb{R}^{n}\times \mathbb{R}^{k}}\big[\prod_{i=1}^{k}\hat{\psi}_{j_{i}}(\tau_{i})\big]\hat{f}(\xi, \tau) e^{2\pi (x \cdot \xi+t \cdot \tau)} \, d\xi d\tau.
$$
Here $\tau=(\tau_1, \ldots, \tau_k)$.  Note that $|\{ J \in \mathbb{Z}^k : f_J \not\equiv 0 \} | = O(\delta^{-\epsilon k/2}) $. Hence, by the triangle inequality 
$$
\|f\|_{L^{p_{0}}(\R^{n+k})} \lesim \delta^{-\epsilon k/2} \max_{J\in \Z^k} \| {f_{J}}\|_{L^{p_{0}}(\mathbb{R}^{n+k})}.
$$
By Proposition \ref{1011prop2.2} and Young's inequality, the last expression can be further bounded by 
\begin{displaymath}
\begin{split}
\delta^{-2\epsilon k} \max_{J\in \R^k} \bigl( \sum_{a \in \mathcal{M}_{\delta}} \|{\Xi_{a} * {f_J}}\|_{L^{p_{0}}(\mathbb{R}^{n+k})}^2 \bigr)^{\frac{1}{2}}
\lesssim_{\epsilon} \delta^{-2\epsilon k} \bigl( \sum_{a \in \mathcal{M}_{\delta}} \| {\Xi_{a} * f}\|_{L^{p_{0}}(\mathbb{R}^{n+k})}^2\bigr)^{\frac{1}{2}}.
\end{split}
\end{displaymath}
Hence, what remains is to show Proposition \ref{1011prop2.2}.

\subsection{}\label{subsection2.2}
Our argument relies on an iteration. This iteration argument first appeared in Pramanik and Seeger \cite{PS}. We will deduce Proposition \ref{1011prop2.2} from the following proposition.
\begin{proposition}\label{1011prop2.3}
	Fix $\mu$ such that $2\mu+\epsilon/2 \leq 1$ and $\mu \geq \epsilon/2$. Let $a \in \mathcal{M}_{\delta^{2\mu}}$. 
	If $\mathrm{supp}(\hat{f}) \subset \mathcal{N}_{\delta}\tilde{S}$, then
	\begin{displaymath}
	\|{\Xi_{a} * {f}}\|_{L^{p_{0}}(\mathbb{R}^{n+k})} 
	\lesssim_{\epsilon}  \delta^{-{\epsilon^3}} 
	\biggl( \sum_{b\in \mathcal{M}_{\delta^{2\mu+\epsilon/2} }} \|{\Xi_{a}*\Xi_{b} * {f}}\|_{L^{p_{0}}(\mathbb{R}^{n+k})}^{2} \biggr)^{\frac{1}{2}}.
	\end{displaymath}
\end{proposition}
We postpone the proof of Proposition \ref{1011prop2.3} to the next subsection, and continue by 
\begin{proof}[Proof of Proposition \ref{1011prop2.2}.]
First of all, by the triangle inequality and H\"older's inequality, we obtain 
\[
\| {f}\|_{p_{0}} \lesssim \delta^{-C\epsilon} \bigl( \sum_{a \in \mathcal{M}_{\delta^{\epsilon}}} \|{\Xi_{a} * f}\|_{p_{0}}^{2}\bigr)^{\frac{1}{2}},
\]
for some large constant $C>0$. Next, by applying Proposition \ref{1011prop2.3} with $\mu=\epsilon/2$, the last expression can be further bounded by 
\begin{displaymath}
\begin{split}
\delta^{-{\epsilon}^3-C\epsilon} 
\bigl(\sum_{a \in \mathcal{M}_{\delta^{\epsilon}}}\sum_{b \in \mathcal{M}_{\delta^{\frac{3\epsilon}{2}}} } \| {\Xi_{a}* \Xi_{b} * {f}}\|			_{p_{0}}^{2}\bigr)^{\frac{1}{2}}
\lesssim_{\epsilon}
\delta^{-{\epsilon}^3-C\epsilon} \bigl(\sum_{b \in \mathcal{M}_{\delta^{\frac{3\epsilon}{2}}} } \| { \Xi_{b} * {f}}\|			_{p_{0}}^{2}\bigr)^{\frac{1}{2}}. 
\end{split}
\end{displaymath}
The last inequality follows from \[|\{ a \in \mathcal{M}_{\delta^{\epsilon}} : \Pi_{a,\delta^{\epsilon}} \cap \Pi_{b,\delta^{3\epsilon/2}} \cap \mathcal{N}_{\delta}S \neq \emptyset \} | = O(1),\] which further follows from Assumption $A_3$ and Assumption $A_4$. Repeatedly apply Proposition 2.4 with $\mu = \mu_{l} = \frac{l}{4}\epsilon$ starting with $l=3$ until $l=\frac{2}{\epsilon}-1$. In the end we have
\begin{displaymath}
\| f\|_{p_{0}} \lesssim_{\epsilon} \delta^{ -2C\epsilon} \bigl( \sum_{a \in \mathcal{M}_{\delta}}\|\Xi_{a}*f \|_{p_{0}}^{2}\bigr)^{\frac{1}{2}}.
\end{displaymath} 
This finishes the proof of Proposition \ref{1011prop2.2}.
\end{proof}

\subsection{}\label{subsection2.3} In this subsection, we will prove Proposition \ref{1011prop2.3} by using Theorem \ref{b-d}. Let $B\subset \R^{n+k}$ be a ball of radius $r_B:=\delta^{-(2\mu+\frac{\epsilon}{2})}$, centered at $c_B$. Let $C$ be a large constant. Define a weight $w_{B}$ associated with the ball $B$ by $(1+\frac{\cdot - c_B}{r_B})^{-C}$. To prove Proposition \ref{1011prop2.3}, by a simple localisation argument, it suffices to prove 
\beq\label{localised}
\|{\Xi_{a} * {f}}\|_{L^{p_{0}}(w_{B})} 
	\lesssim_{\epsilon}  \delta^{-{\epsilon^3}} 
	\biggl( \sum_{b\in \mathcal{M}_{\delta^{2\mu+\epsilon/2} }} \|{\Xi_{a}*\Xi_{b} * {f}}\|_{L^{p_{0}}(w_B)}^{2} \biggr)^{\frac{1}{2}}.
\endeq
Let $a \in \eta(y_{0},\ldots,y_{k})$ for some good $(k+1)$-tuple $(y_{0},\ldots,y_{k})$. Under certain translation and rotation, we may assume that $y_0$ lies in the origin and 
$$
y_i=(\underbrace{0, \ldots, 0}_{n \text{ copies }}, \underbrace{0, \ldots, 0}_{i-1 \text{ copies }}, 1, 0, \ldots, 0) \text{ for each } 1\le i\le k.
$$
Moreover, we assume that the normal vector to the surface $S$ at the point $y_i$ is given by $e_n$ for every $0\le i\le k$.
%
By using a partition of unity, we may assume, that each $E_i$, viewed as a hypersurface in $L_i$, can be represented as the graph of a smooth function $G_{i}: (-\epsilon_0, \epsilon_0)^{n-1}\to \R$ for some small constant $\epsilon_0>0$ that might vary from line to line. Under these assumptions, we observe that $\nabla G_i(0)=(0, \ldots, 0)\in \R^{n-1}$ for each $i$. Moreover, for those points that are different from the origin, we have 
\begin{claim}\label{claim}
For each $1\le i\le k$, there exists a smooth function $h_{i}:(-\epsilon_0,\epsilon_0)^{n-1} \to \mathbb{R}^{n-1}$ such that 
$$
\nabla G_{i}(h_{i}(\xi'))= \nabla G_{0}(\xi') \text{ for all }\xi' \in (-\epsilon_0,\epsilon_0)^{n-1}.
$$
Moreover, $h_{i}(\xi')=J_i\cdot \xi' +O(|\xi'|^2)$ for some positive definite matrix $J_i$. 
\end{claim}
\begin{proof}
For each $\xi'$, let us consider the level set $\{\eta'\in \R^{n-1}: \nabla G_i(\eta')=\nabla G_0(\xi')\}$. Note that this set is not empty. Recall that $G_i$ is a strictly convex smooth function. Hence the existence and smoothness of $h_i$ can be guaranteed by the implicit function theorem.

To obtain an asymptotic of the function $h_i$ near the origin, we differentiate both sides of the equation $\nabla G_{i}(h_{i}(\xi'))= \nabla G_{0}(\xi')$, and obtain $(HG_{i})(\nabla h_{i}) = HG_{0}$. Here $HG_{i}$ is the Hessian matrix of the function $G_{i}$. Since $E_{i}$ is strictly convex, $HG_{i}$ is a positive definite matrix. Thus, $\nabla h_{i}$ is also a positive definite matrix. The identity $h_{i}(\xi')=J_i\cdot \xi'+O(|\xi'|^2)$, with some positive definite matrix $J_i$, immediately follows from Taylor's theorem. This completes the proof of the claim.  
\end{proof}

Denote $h_0(\xi')=\xi'$. By Claim \ref{claim}, if $\epsilon_0$ is chosen small enough, then a good $(k+1)$-tuple containing $(\xi', G_0(\xi'), 0, \ldots, 0):=P_0(\xi')$ also contains 
$$
(h_i(\xi'), G_i(h_i(\xi')), \underbrace{0, \ldots, 0}_{i-1 \text{ copies }}, 1, 0, \ldots, 0):=P_i(\xi').
$$
for each $1\le i\le k$. Hence, w.l.o.g. we may assume that the $k$-cone $\tilde{S}\cap \Pi_{a, \delta^{2\mu}}$ is given by 
\begin{equation}\label{cone-part}
\begin{gathered}
\{(1-\sum_{j=1}^{k}\theta_{j})P_0(\xi')+ \sum_{j=1}^{k}\theta_{j} P_j(\xi') : |\xi'| \lesssim \delta^{\mu} ,\; 0 \leq \theta_{j} \leq \delta^{\epsilon/2}	\}.
\end{gathered}
\end{equation}
We claim that  the $k$-cone given by \eqref{cone-part} is contained in the $\delta^{2\mu+\frac{\epsilon}{2}}$ neighbourhood of a cylinder. To be precise, we will use the cylinder
\beq\label{cylinder}
\{P_0(\xi')+\sum_{j=1}^k \theta_j e_{n+j}: |\xi'|\lesim \delta^{\mu}, 0\le \theta_j\le \delta^{\epsilon/2}\}.
\endeq
That is, we will show that, given an arbitrary point on the $k$-cone \eqref{cone-part}, its distance with the cylinder \eqref{cylinder} is smaller than $\delta^{2\mu+\frac{\epsilon}{2}}$. Given a point in \eqref{cone-part}, we write it as 
$$
((1-\sum_{j=1}^k \theta_j) \xi'+\sum_{j=1}^k \theta_j h_j(\xi'), (1-\sum_{j=1}^k \theta_j) G_0(\xi')+\sum_{j=1}^k \theta_j G_j(h_j(\xi')), \theta_1, \ldots, \theta_k).
$$
We calculate its distance with the point 
$$
\Big((1-\sum_{j=1}^k \theta_j) \xi'+\sum_{j=1}^k \theta_j h_j(\xi'), G_0\Big((1-\sum_{j=1}^k \theta_j) \xi'+\sum_{j=1}^k \theta_j h_j(\xi')\Big), \theta_1, \ldots, \theta_k\Big)
$$
from the cylinder \eqref{cylinder}. This amounts to proving 
$$
\Big|G_0\Big((1-\sum_{j=1}^k \theta_j) \xi'+\sum_{j=1}^k \theta_j h_j(\xi')\Big)-(1-\sum_{j=1}^k \theta_j) G_0(\xi')+\sum_{j=1}^k \theta_j G_j(h_j(\xi'))\Big| \lesim \delta^{2\mu+\frac{\epsilon}{2}}.
$$
By the triangle inequality, it suffices to show 
$$
\Big|G_0\Big((1-\sum_{j=1}^k \theta_j) \xi'+\sum_{j=1}^k \theta_j h_j(\xi')\Big)-G_0(\xi')\Big| \lesim \delta^{2\mu+\frac{\epsilon}{2}}
$$
and 
$$
|G_0(\xi')-G_j(h_j(\xi'))| \lesim \delta^{2\mu}.
$$
The latter follows directly from Taylor's formula. To prove the former estimate, we write $G_0(\xi')=(\xi')^T [HG_0(0)] \xi'+O(|\xi'|^3)$ with $HG_0(0)$ the Hessian matrix of the function $G_0$ at the origin. Moreover, we know that $HG_0(0)$ is positive definite. Using this formula, we just need to show that 
$$
\Big|(1-\sum_{j=1}^k \theta_j) \xi'+\sum_{j=1}^k \theta_j h_j(\xi')-\xi'\Big| \lesim \delta^{\mu+\frac{\epsilon}{2}},
$$
which follows via a direct calculation. \\

So far we have verified that the $k$-cones \eqref{cone-part} lies in a $\delta^{2\mu+\frac{\epsilon}{2}}$-neighbourhood of the cylinder \eqref{cylinder}. Hence to prove the localised decoupling inequality \eqref{localised}, by the uncertainly principle, it is the same as proving a corresponding decoupling inequality associated with the cylinder \eqref{cylinder}, which further follows from Theorem \ref{b-d} and Fubini's theorem. This finishes the proof of Proposition \ref{1011prop2.3}.

\section{Proof of Theorem 1.2}\label{sec3}

In this section, we use the notations defined at the beginning of Section \ref{sec2}. The proof of Theorem \ref{main2} essentially follows via the same argument as that of Theorem \ref{main1}. Hence we will only write down the relevant estimates and omit most of the details. \\

To prove Theorem \ref{main2}, we will use the following theorem due to Bourgain and Demeter.

\begin{theorem}[\cite{BD1}]\label{b-d2}
	Denote $p_{0}=\frac{2(n+1)}{n-1}$. Fix $v \in (\mathbb{R}\setminus \{0\})^{n-1}$. If $\mathrm{supp}(\hat{f}) \subset \mathcal{N}_{\delta}H_{v}^{n-1}$, then 
	\begin{equation}\label{eq b-d2}
	\|f\|_{L^{p_{0}}(\mathbb{R}^{n})} \lesssim_{\epsilon} \delta^{\frac{n}{p_{0}}-\frac{n-1}{2}-\epsilon}\biggl(\sum_{a \in \mathcal{M}_{\delta}} \| \Xi_{a}' * f \|_{L^{p_{0}}(\mathbb{R}^{n})}^{p_{0}}\biggr)^\frac{1}{p_{0}},
	\end{equation}
	for every $\epsilon>0$.
\end{theorem}

The role of Theorem \ref{b-d2} in the proof of Theorem \ref{main2} is similar to that of Theorem \ref{b-d} in the proof of Theorem \ref{main1}.
However, in contrast with the proof of Theorem \ref{main1}, we need a rescaled version of Theorem \ref{b-d2}. This is because the exponent of $\delta$ in \eqref{eq b-d2} is not arbitrarily small, which requires us to carefully deal with the exponent of $\delta$ there.

By performing simple parabolic rescaling to Theorem \ref{b-d2}, we have the following proposition. The interested
reader should consult the proof of Propositon 4.1 in \cite{BD2} for details.

\begin{proposition}\label{rescalied b-d2}
	Denote $p_{0}=\frac{2(n+1)}{n-1}$. Fix $v \in (\mathbb{R} \setminus 0)^{n-1}$ and $\alpha,\mu>0$. If $\mathrm{supp}(\hat{f}) \subset \mathcal{N}_{\delta^{2\mu+\alpha}}H_{v}^{n-1} \cap (\{(\xi_{1},\ldots,\xi_{n-1}) : |\xi_{i}| \leq \delta^{\mu}\} \times \mathbb{R})$, then
	$$\|f\|_{L^{p_{0}}(\mathbb{R}^{n})} \lesssim_{\epsilon}
	\delta^{\alpha(\frac{n}{p_{0}}-\frac{n-1}{2})-\epsilon}\biggl(\sum_{a \in \mathcal{M}_{\delta^{2\mu+\alpha}}} \| \Xi_{a}' * f \|_{L^{p_{0}}(\mathbb{R}^{n})}^{p_{0}}\biggr)^\frac{1}{p_{0}},
	$$
	for every $\epsilon>0$.		
\end{proposition}

The forthcoming proof of Theorem \ref{main2} is  similar to the one in Section \ref{sec2}.  As we did in Section \ref{sec2}, by interpolation, it suffices to consider only the endpoint $p_{0}=\frac{2(n+1)}{n-1}$.

\subsection{}
In the first step of the proof, we will slice our surface into small pieces so that we can exploit local properties of the conical surface. 
By a linear transformation, we may assume that $L_{0}=\mathbb{R}^d \times \{1\}$ and $C_{1}=1$. 

Fix a small parameter $\epsilon>0$. This $\epsilon$ is essentially the same as the one in the statement of Theorem \ref{main2}. We may also assume that $\epsilon^{-1}$ is a natural number.
We define a sliced surface $\tilde{S}$ by
$$
\tilde{S}= S \cap (\mathbb{R}^n \times \{ \tau_{1} : d \leq \tau_{1} \leq d+4\delta^{\epsilon/2}	\})
$$
for some $d$.
We will prove the decoupling for the sliced surface $\tilde{S}$ first.

\begin{proposition}\label{prop3.3}
	If $\mathrm{supp}(\hat{f}) \subset \mathcal{N}_{\delta}\tilde{S}$, then
	$$\| f \|_{L^{p_{0}}(\mathbb{R}^{n+1})} \lesssim_{\epsilon} \delta^{\frac{n}{p_{0}}-\frac{n-1}{2}-\epsilon}\biggl( \sum_{a \in \mathcal{M}_{\delta}} \|{\Xi_{a} * f}\|_{L^{p_{0}}(\mathbb{R}^{n+1})}^{p_{0}} \biggr)^{\frac{1}{p_{0}}}.  $$
\end{proposition}

The desired decoupling inequalities for the surface $S$ can be deduced from Proposition \ref{prop3.3}. This can be shown by using arguments in Subsection \ref{subsection2.1}, so we do not reproduce it here. Hence, what remains is to show Proposition \ref{prop3.3}.

\subsection{} 
We will deduce Proposition \ref{prop3.3} from the following proposition.
\begin{proposition}\label{prop3.4}
	Fix $\mu >0$ such that $2\mu+\epsilon/2 \leq 1$ and $\mu \geq \epsilon/2$. Let $a \in \mathcal{M}_{\delta^{2\mu}}$. 
	If $\mathrm{supp}(\hat{f}) \subset \mathcal{N}_{\delta}\tilde{S}$, then
	\begin{displaymath}
	\|{\Xi_{a} * {f}}\|_{L^{p_{0}}(\mathbb{R}^{n+1})} 
	\lesssim_{\epsilon}  \delta^{\frac{\epsilon}{2}(\frac{n}{p_{0}}-\frac{n-1}{2}-\epsilon)} 
	\biggl( \sum_{b\in \mathcal{M}_{\delta^{2\mu+\epsilon/2} }} \|{\Xi_{a}*\Xi_{b} * {f}}\|_{L^{p_{0}}(\mathbb{R}^{n+1})}^{p_{0}} \biggr)^{\frac{1}{p_{0}}}.
	\end{displaymath}
\end{proposition}

Proposition \ref{prop3.3} can be deduced from Proposition \ref{prop3.4} by arguments in Subsection \ref{subsection2.2}, so we omit the details here.

\subsection{} In this subsection, we will deduce Proposition \ref{prop3.4} from Proposition \ref{rescalied b-d2}. To do this, we will just follow the arguments used in Subsection \ref{subsection2.3}.
Let $B\subset \R^{n+1}$ be a ball of radius $r_B:=\delta^{-(2\mu+\frac{\epsilon}{2})}$, centered at $c_B$. Let $C$ be a large constant. Define a weight $w_{B}$ associated with the ball $B$ by $(1+\frac{\cdot - c_B}{r_B})^{-C}$. To prove Proposition \ref{prop3.4}, by a simple localisation argument, it suffices to prove 
\beq\label{localised2}
\|{\Xi_{a} * {f}}\|_{L^{p_{0}}(w_{B})} 
\lesssim_{\epsilon}   \delta^{\frac{\epsilon}{2}(\frac{n}{p_{0}}-\frac{n-1}{2}-\epsilon)} 
\biggl( \sum_{b\in \mathcal{M}_{\delta^{2\mu+\epsilon/2} }} \|{\Xi_{a}*\Xi_{b} * {f}}\|_{L^{p_{0}}(w_B)}^{p_{0}} \biggr)^{\frac{1}{p_{0}}}.
\endeq
Under certain linear transformation, we may assume that
$a=(0,\ldots,0,1)$ and 
 $E_{0}$ is represented as the graph of a smooth function $G$ with $G(0)=0$ and $\nabla G(0)=(0,\ldots,0) \in \mathbb{R}^{n-1}$. Hence, 
w.l.o.g we may assume that the conical surface $\tilde{S} \cap \Pi_{a,\delta^{2\mu}}$ is given by
\beq\label{var}
\{(1+\theta)(\xi',G(\xi'),1) \in \mathbb{R}^{n-1} \times \mathbb{R} \times \mathbb{R} : |\xi'| \lesssim \delta^{\mu}, 0 \leq \theta \leq \delta^{\epsilon/2} \}.
\endeq
We claim that this surface is contained in the $\delta^{2\mu +\epsilon/2}$ neighborhood of the following cylinder
\beq\label{conic cyl}
\{(\xi',G(\xi')) : |\xi'| \lesssim \delta^{\mu}	\} \times \mathbb{R}.
\endeq
To see this, we take any point in \eqref{var}, and we write it as
$$
\Big(	(1+\theta)\xi',(1+\theta)G(\xi'),1+\theta
\Big).
$$
We calculate its distance with the point
$$
\Big( (1+\theta)\xi', G((1+\theta)\xi'),1+\theta
\Big)
$$
from the cylinder \eqref{conic cyl}. This amounts to proving
$$
\Big|(1+\theta)G(\xi')-G((1+\theta)\xi')\Big| \lesim \delta^{2\mu+\frac{\epsilon}{2}},
$$
which follows directly from Taylor's formula.
\\

So far we have verified that the conical surfaces \eqref{var} lies in a $\delta^{2\mu+\frac{\epsilon}{2}}$-neighbourhood of the cylinder \eqref{conic cyl}. Hence to prove the localised decoupling inequality \eqref{localised2}, by the uncertainly principle, it is the same as proving a corresponding decoupling inequality associated with the cylinder \eqref{conic cyl}, which further follows from Proposition \ref{rescalied b-d2} and Fubini's theorem. This finishes the proof of Proposition \ref{prop3.4}.

\noindent Department of Mathematics, Indiana University, 831 East 3rd St., Bloomington IN 47405\\
\emph{Email address}: shaoguo@iu.edu\\

\noindent Department of Mathematics, Pohang University of Science and Technology, Pohang 790-784, Republic of Korea\\
\emph{Email address}: ock9082@postech.ac.kr

\end{document}